\documentclass{amsart}

\usepackage{macros}

\title[Pinned distances in positive characteristic]{Bisector energy and pinned distances in positive characteristic}

\author{Brendan Murphy}
\address{Department of Mathematics, University of Bristol, Bristol BS8 1UG, UK}
\email{brendan.murphy@bristol.ac.uk}
\author{Misha Rudnev}
\address{Department of Mathematics, University of Bristol, Bristol BS8 1UG, UK}
\email{misha.rudnev@bristol.ac.uk}
\thanks{The first and second authors are partially supported by the Leverhulme  Trust Grant RPG--2017--371}

\author{Sophie Stevens}
\address{Department of Mathematics, University of Bristol, Bristol BS8 1UG, UK}
\email{sophie.stevens@bristol.ac.uk}

\date\today

\begin{document}
\begin{abstract}
    We prove a new lower bound for the number of pinned distances over finite fields: if $A$ is a sufficiently small subset of $\F_q^2$, then there is an element in $A$ that determines $\gg |A|^{2/3}$ distinct distances to other elements of $A$.
    
    In fact, we obtain an upper bound for the number of isosceles triangles determined by $A$.
    For that we use the concept of bisector energy.  It turns out that the latter can be expressed as a point-plane incidence bound, so one can use a theorem of the third author.
    
    The conversion to this incidence problem relies on the Blaschke-Gr\"unwald kinematic mapping -- an embedding of the group of rigid motions of $\F_q^2$ into an open subset of the projective three space.
    This has long been known in kinematics and geometric algebra; we provide a proof for arbitrary fields using Clifford algebras.
\end{abstract}

\maketitle
\setcounter{tocdepth}{1}
\tableofcontents

\section{Introduction}
    In 1946, Erd{\H{o}}s \cite{erdos1946distances} conjectured that any set of $N$ points in the real plane determines at least $\gg \frac{N}{\sqrt{\log N}}$ distinct distances, with a square grid showing this bound is optimal.
    Guth and Katz \cite{guth2015erdos} nearly resolved this problem 65 years later, proving that $N$ points determine $\gg \frac{N}{\log{N}}$ distances. 
    
    One can also ask this question for non-Euclidean distances over the reals, where some metrics behave similar to the Euclidean one but others do not. 
    For instance, Roche-Newton \cite{roche-newton2012minkowski} and Selig and the second author \cite{rudnev2016klein} prove positive results, while Matou\v{s}ek \cite{matouvsek2011number} and Valtr \cite{valtr2005strictly} prove negative results (see also the earlier work \cite{iosevich2007distance} by Iosevich and the second author in the Falconer conjecture context, as well as the recent paper \cite{guth2018falconer} by Guth et al.).
    For more relatives of the distinct distances problem over the reals we recommend the survey of Sheffer \cite{sheffer2014distinct}.
    
    The distinct distance problem can also be posed over arbitrary fields $\F$, where much less is known about the tools used and developed by Guth and Katz.
    Setting up the notation, the distance $d$ between two points $x=(x_1,x_2), y=(y_1,y_2)$ in the plane $\F^2$ is
    \[
        d\left(x,y\right):= (x-y)\cdot(x-y) =\|x-y\|^2 = (x_1-y_1)^2+ (x_2-y_2)^2.
    \]
    Typically, $\F$ will be a finite field of order $q$, where $q=p^r$ is an odd prime power, though some results hold in arbitrary fields $\F$ with constraints in terms of the characteristic of $\F$.
    
    Distinct distances correspond to equivalence classes of pairs of points in $\F^2$ modulo the action of rigid motions.
    We want to find a lower bound on the cardinality $\Delta(A)$ of the set of distances determined by a finite point set $A\subset \F^2$, where we define $\Delta(A):=|\{d(a,b)\colon a,b\in A\}|$.

    A sensible interpretation of the distance problem over a general field $\F$ may be more subtle than over the reals.
    If $\F$ is finite, then $\Delta(A)\leq |\F|$ for any $A$, so $\Delta(A)\gg |A|/\log |A|$ cannot hold unconditionally.
    In fact, if $\F$ has characteristic $p$, then there are subsets $A\subseteq\F^2$ with $\Delta(A)\leq p$.
    For this reason we constrain the cardinality of $A$ in terms of the characteristic $p$ of $\F$ if $p>0$. 
    Another issue is that of \emph{isotropic lines}\/ in the plane: if $A$ contains points exclusively of the form $(a,{\rm i}a)$ in $\C^2$, then the only distance between pairs of points of $A$ is $0$.
    This may happen over finite fields $\F_q$ with $q\equiv 1\bmod{4}$.
    
    Turning our first obstruction on its head, we expect that if $A\subseteq \F^2_q$ is sufficiently large relative to $q$, then all, or at least a positive proportion of possible distances  should be determined.
    Determining the threshold at which this occurs is often referred to as the Erd\H{o}s-Falconer problem \cite{iosevich2007distance}.
    Iosevich and the second author \cite{iosevich2007distance} initiated this point of view, asking for the corresponding lower bounds on $|A|$.
    Some obstructions were identified by Hart et al. \cite{hart2011averages-a} and Petridis and the first author \cite{MP_Falc}: when $|A| \leq q^{4/3}$, subfields may preclude $|A| > q/2$.
    For other recent developments in this direction we refer the reader to Koh, Pham and Vinh \cite{koh2018distance} and references therein.
    
    The first result on the Erd\H{o}s distinct distance problem over $\F_q$ was obtained by  Bourgain, Katz and Tao \cite{bourgain2005sum-product}, who proved a non-quantitative non-trivial bound on $\Delta(A)$, based on a non-trivial Szemer\'edi-Trotter type theorem, which in turn followed from their sum-product estimate.
    A strong quantitative variant of this theorem, due to  de Zeeuw and the third author \cite{stevens2017improved}, implies the bound $\Delta(A)\gg |A|^{8/15}$ under suitable conditions on $|A|$.
    Iosevich, Koh and Pham \cite{iosevich2019perspective} strengthen this technique by using bounds on the additive energy of a set lying on a paraboloid, improving the exponent to $\frac{1}{2} + \frac{69}{1558}$.
    This result is valid for $\F=\F_p$ and $p\equiv 3\bmod{4}$.
    
    The distinct distance problem has a stronger variant, known as the pinned distance problem. 
    The pinned distance problem asks for the existence of a point $a\in A$, from which many distinct distances are determined. It is open over the reals as well, the last series of progress having been achieved back in the early 2000s, with the standing record by Katz and Tardos \cite{katz-tardos}.
    
    This note proves new bounds on the pinned distance problem over a general field $\F$, building on work by Lund and Petridis~\footnote{We are grateful to Ben Lund and Giorgis Petridis for clarifications of their results and discussions throughout the preparation of this note.}\cite{lund2018pinned}, which further developed the earlier approach \cite{chapman2012pinned,hanson2016distinct} of studying perpendicular bisectors to achieve the following result.
    
    \begin{thm}[Lund-Petridis \cite{lund2018pinned}]\label{t:lp}
        Let $\F$ be a field, and $A\subseteq \F\times \F$. 
        If the characteristic of $\F$ is $p>0$, suppose also that  $|A|\leq p^{8/5}$.
        
        Then there exists $a\in A$ such that $|\{d(a,b)\colon b\in A\}| \gg |A|^{20/37}$, and in particular, $\Delta(A) \gg |A|^{20/37}$.
    \end{thm}
    We also note that Petridis \cite{petridis2016pinned} has proved a stronger result, on the assumption that $A\subseteq \F\times \F$ is a Cartesian product $A=X\times X$, for any $X\subseteq\F$ satisfying $|X|\ll \mathrm{char}(\F)^{2/3}$.
    This result can be proved directly from Rudnev's point-plane incidence bound \cite{rudnev2018point-plane}, or by descendants of this bound \cite{yazici2015growth,murphy2017second}.

\subsection{Main Result}
    We prove a lower bound on the number of pinned distances $\pind(A)$ determined by $A$, where 
    \[
    \pind(A):= \max_{a\in A}|\Delta(A,a)|,
    \]
    where $\Delta(A,a):=\{d(a,b)\colon b\in A\}|$.
    
    It is clear that $\Delta(A) \geq \pind(A)$.  With this notation, Theorem~\ref{t:lp} can be restated as $\pind(A)\gg |A|^{20/37}$. We now present the main results of this paper. 
    
    \begin{thm}\label{t:main-general}
        Let $A\subset \F^2$ be a set of points, at most a third of which lie on a single isotropic subspace.
        If the characteristic of $\F$ is $p>0$, suppose also that  $|A|\leq p^{4/3}$.
        Then
        \begin{equation}\label{e:distances-general}
            \pind(A) \gg |A|^{2/3}.
        \end{equation}
    \end{thm}
    Theorem~\ref{t:main-general} is an immediate consequence of Proposition~\ref{prop:isosceles}, below, and a Cauchy-Schwarz argument.
    
    Unfortunately, there is still a gap between the estimates of Theorem \ref{t:main-general} and what is known for sufficiently large sets $A$ for the Erd\H os-Falconer problem over finite fields $\F_q$ ($p$ and $q$ are always assumed to be odd and sufficiently large). Chapman et al. \cite[Theorem 2.2]{chapman2012pinned}  used Fourier analysis to show that for any $|A|\gg q^{4/3}$, $\Delta(A)\gg q$. The latter paper claimed the bound for $q \equiv 3\bmod{4}$ only; it was then observed by  Bennett et al. \cite[Theorem 1.6]{bennett2017group} that the same proof (which is, in fact, a variant of \cite[Proposition 4.29]{tao2010additive}) works for $q \equiv 1\bmod{4}$ as well.
    Furthermore, Hanson, Lund and Roche-Newton \cite{hanson2016distinct} extended the claim to $\pind(A)$, using spectral graph theory, applied to studying perpendicular bisectors, the initial set-up being similar to that of Theorems \ref{t:lp}, \ref{t:main-general}.

\subsection{Discussion of techniques}
    To prove Theorem \ref{t:main-general}, we follow the established technique of studying distances through the perpendicular bisectors of points in $A$.
    The perpendicular bisector of points $a,b\in \F^2$ with $d(a,b)\neq 0$ is the line
    \[
        \bisector(a,b) =\{x\in \F^2 \colon d(a,x) = d(b,x)\}.
    \]
    We consider (a subtle variant of) the \emph{bisector energy}\/ of the set $A$, which is the number of pairs of points whose perpendicular bisector coincide:
    \[
        \B(A):=|\{(a,b,c,d)\in A^4\colon\bisector(a,b) = \bisector(c,d)\}|.
    \]
    The variant of the bisector energy that we use --- $\B^*(A)$ in Section~\ref{sec:bisect-energy-isosc} --- allows us to disregard the delicacies that arise from isotropic lines.
    
    The bisector energy controls the number of \emph{isosceles triangles}\/ in $A$, and upper bounds for the number of isosceles triangles in $A$ yield lower bounds for $\pind(A)$.
    Lund and Petridis show quantitatively that if the bisector energy is large, then $A$ contains many collinear points or many co-circular points \cite[Theorem 2]{lund2018pinned}, thus either the number of isosceles triangles is small, or the set $A$ has structure.
    
    To prove our bound on the modified bisector energy $\B^*(A)$, we use the \emph{kinematic mapping}\/ of Blaschke and Gr\"unwald to embed the space of segments of the same length into projective three-space.
    The bisector energy in a class of $n$ segments of the same length is then represented by the number of incidences between $n$ points and $n$ planes, which we bound using the point-plane incidence theorem of the second author.
    
    To be precise, we use $S_r=S_r(A)$ to denote the set of pairs of points of distance $r$ apart:
    \[
      S_r = S_r(A) :=\{(a,b)\in A^2\colon d(a,b)=r\}.
    \]
    The modified bisector energy $\B^*(A)$ is equal to the sum over $r\not=0$ of the number of pairs of segments in $S_r(A)$ that are \emph{axially symmetric}\/ (plus an error term for isotropic segments).
    As mentioned above, we count the number of such pairs by representing it as a point-plane incidence count in projective three-space; see Claim 1 below.
    From this it follows that
    \[
        \B^*(A)\ll \sum_{r\not=0}|S_r|^{3/2} \ll |A| \left( \sum_{r\not=0}|S_r|^2 \right)^{1/2},
    \]
    unless $A$ has many collinear or co-circular points.
    (In the proof, we assume without loss of generality that $\F$ is algebraically closed, since we may embed $A$ into the algebraic closure of $\F$ without decreasing the quantities we wish to bound.)

\subsection*{Notation} 
    We use standard asymptotic notation: $f(n)\ll g(n)$ or $f=O(g)$ means that there exists a constant $c>0$ that does not depend on $n$ so that $f(n)\leq c|g(n)|$.
    The relation $f\gg g$ is equivalent to $g\ll f$.
    The constant implicit in this notation may freely change from line to line. 
    
    We $p$ to denote an odd prime, and we use $\F_q$ to denote a field of prime power cardinality, with $q=p^r$ for some $r>0$.
    We often simply write $\F$ for a field, which is often, but not necessarily $\F_q$.
    Most hypotheses involving a set $A\subseteq\F_q^2$ include a constraint in terms of the characteristic $p$; for instance $|A|\leq p^c$, for some $c>0$.
    We always state the this constraint as a $\leq$ bound, noting that if it happens to be satisfied up to an absolute constant, the only effect would be a change in the implicit constants in the conclusion.
    
    Also, for results pertaining to finite fields $\F_p$ and $\F_q$, the prime $p$ is to be sufficiently large to dominate the suppressed universal constants.

\section{Preliminaries}
\subsection{Distance preserving transformations}
    Let $\SO\subseteq\SL$ denote the set of unit determinant linear transformations preserving the distance:
    \[
        \SO:=\{g\in \SL\colon \forall x,y\in\F^2,\,d(x,y)=d(gx,gy)\}.
    \]
    As a matrix group,
    \[
        \SO  =\left\{
        \begin{pmatrix}
            u & - v \\
            v & u\\
        \end{pmatrix}:
        u,v\in \F, u^2+ v^2=1 \right\}.
    \]
    We will use the notation $\Cc\subseteq \F^2$ for the unit circle, and write $(u,v)\in \Cc$.
    As is the case for rotations acting on circles in $\R^2$, the group $\SO$ acts simply transitively on the level sets $\{(x,y)\in\F^2\colon d(x,y)=t\}$ for all $t\not=0$.
    Thus $d(x,y)=d(x',y')$ if and only if there is a rotation $\theta\in \SO$ such that $\theta x - \theta y = x'-y'$.
    
    Let $T_2(\F)$ be the group of translations $x\mapsto x+t$ acting on the plane $\F^2$. 
    The group $\SF$ of positively oriented rigid motions of $\F^2$ is generated by $\SO$ and $T_2(\F)$; this is the analogue of the special Euclidean group $\mathrm{SE}_2(\R)$. 
    
    Explicitly, an element of $\SF$ has the form: 
    \begin{equation}\label{matrix}
      \begin{pmatrix}
        u&- v&s\\
        v&u&t\\
        0&0&1
           \end{pmatrix},
       \quad \mbox{where $u^2+ v^2 = 1$}.
    \end{equation}
    
    By the above discussion, we see that $d(x,y)=d(x',y')$ if and only if there exists $g\in \SF$ such that $g(x,y)=(x',y')$.
    If such a $g$ exists, an easy calculation shows that it is unique.

\subsection{Blaschke-Gr\"unwald Kinematic Mapping}
    The Blaschke-Gr\"unwald kinematic mapping \cite{blaschke1960kinematik, Grunwald1911abbildungsprinzip} assigns to each element
    $g\in \mathrm{SE}_2(\R)$ %, given by a matrix \eqref{matrix} -- a composition of a rotation and translation --
    a point in projective space $\mathbb{P}\R^3$.
    For a detailed exposition concerning this mapping and its properties, see the textbook by Bottema and Roth \cite[Chapter 11] {bottema1990theoretical}.
    The kinematic mapping was rediscovered some 100 years later by Elekes and Sharir \cite{elekes2011incidences} and played an essential role in the resolutions of the Erd\H{o}s distinct distance problem in $\R^2$ by Guth and Katz \cite{guth2015erdos}.
    
    The definition of the original \BG{} kinematic mapping extends to all fields that are closed under taking square roots.
    The reason for this is the necessity to have well-defined  ``half-angles'': for all $(u,v)\in \mathcal C$ (the unit circle), we may resolve the system of quadratic equations
    \begin{equation}\label{trig}
        u = \tilde u^2 - \tilde v^2,\qquad v = 2\tilde u \tilde v
    \end{equation}
    to find another point $(\tilde u, \tilde v)\in \mathcal C$.
    Since we use projective coordinates, it does not matter which of the two roots of the equation $\tilde u^2 = \frac{1+u}{2}$ one chooses for $\tilde u$, and this choice, once made, defines $\tilde v$ unambiguously.
    With these preliminaries in hand, we may define the Blaschke-Gr\"unwald kinematic mapping, which embeds $\SF$ in $\PF$: an element of $\SF$ of the form of \eqref{matrix} becomes the projective point:
    \begin{equation}\label{bgr}
        [X_0:X_1:X_2:X_3] = [2 \tilde u: 2 \tilde v: s  \tilde u  +  t \tilde v : s \tilde v - t  \tilde u].
    \end{equation}
    Note that the mapping  \eqref{bgr} does not depend on the sign choice in the half-angle formulae \eqref{trig}.
    Conversely,
    \begin{equation}\label{bgrinv}
        u = \frac{X_0^2-X_1^2}{X_0^2+X_1^2},\; v=  \frac{2X_0X_1}{X_0^2+X_1^2} ,\; \frac{s}{2} = \frac{X_1X_3+X_0X_2}{X_0^2+X_1^2}, \; \frac{t}{2}=\frac{X_1X_2-X_0X_3}{X_0^2+X_1^2}.
    \end{equation}
    
    If $\F$ is a field where some elements do not have square roots, we can use projectivity to avoid them.
    If $\tilde u\not=0$, we may multiply the coordinates of the left hand side of \eqref{bgr} to find
    \[
        [X_0:X_1:X_2:X_3] = [2(u+1) : 2v : s(u+1)+  tv : s v - t(u+1)].
    \]
    If $\tilde u=0$, then $u=-1$ and $\tilde v=\pm 1$, so the formula in the previous equation is still correct.

    Observe that the image of the kinematic mapping $\kappa$, is $\mathbb{PF}^3\setminus \{X_0^2+X_1^2=0\}$. That is one removes from $\mathbb{PF}^3$ the exceptional set, which is a line if $-1$ is not a square and is a union of two planes if $-1$ is a square.
    
    The kinematic mapping has a number of remarkable properties, however, the easiest way to derive these properties is by studying a certain Clifford algebra.
    Since we do not have a reference for these computations over arbitrary fields, we carry them out in Appendix~\ref{sec:clifford-calc}.
    
    The most important property of $\kappa$ for this paper is that translation in the group $\SF$ corresponds to a \emph{projective transformation}\/ of \PF.
    \begin{prop}
      \label{prop:kinematic-equivariant}
      For all $g\in\SF$ there are projective maps $\phi_g\colon \PF\to \PF$ and $\phi^g\colon \PF\to \PF$ such that for all $x\in\SF$
      \[
    \kappa(gx) =\phi_g(\kappa(x))\andd \kappa(xg)=\phi^g(\kappa(x)).
        \]
    \end{prop}
    The proof of this proposition is part of Corollary~\ref{cor:1} in Appendix~\ref{sec:clifford-calc}.
    
    As a corollary, we see that the set of all rigid motions mapping one fixed point to another fixed point corresponds to a projective line.
    For points $x$ and $y$ in $\F^2$, let $T_{xy}$ denote the set of $g\in\SF$ such that $gx=y$.
    \begin{cor} \label{cor:transporters}
        For all $x$ and $y$ in $\F^2$, the image $\kappa(T_{xy})$ is a projective line.
    \end{cor}
    \begin{proof}
        The image of the rotation subgroup $\SO$ is $X_2=X_3=0$, which is a projective line.
        By the transformation properties, all conjugate subgroups of $\SO$ are projective lines, and all cosets of these groups are lines.
        But the set $T_{xy}$ is a left coset of the stabiliser of $x$, which is conjugate to $\SO$.
    \end{proof}

\subsection{Isotropic lines} \label{sec:isotropy}
    In arbitrary fields, there may exist a set of (isotropic) points whose pairwise distance is 0. This is an obvious obstruction to obtaining a lower bound on $\pind(A)$, and so we have to consider these points separately. 
    
    A vector $v\neq 0\in \F^2$ is \emph{isotropic} if $d(v,v)=0$. If ${\rm i}:=\sqrt{-1}\in \F$, then $\F$ contains isotropic vectors. In particular, we note that when $p\equiv 3\mod(4)$, then $-1$ is not a square so there are no isotropic vectors. 
    Give a finite point set $A$, an oriented segment is a pair $(a,a')\in A^2$ with length $d(a,a')$. If $d(a,a')=0$, the segment is called isotropic; when $a\neq a'$ we say that $(a,a')$ is a non-trivial isotropic segment. A non-trivial isotropic segment lies on an isotropic line with slope $\pm {\rm i}$.
    
    Isotropic line segments should be excluded from counts, for there may be too many of them: a single isotropic line supporting $N$ points contains $\gg N^2$ zero-length segments. 
    
    Among other facts on isotropic lines, it's easy to see that perpendicular bisectors are not isotropic  \cite[Corollary 8]{lund2018pinned}.

\subsection{Axial Symmetries}\label{sec:axial}
    As in the Euclidean case, $\SF$ has index two in the group of all distance-preserving transformations.
    The other coset of $\SF$ consists of compositions of reflection through some (non-isotropic) line, and a translation parallel to this line.
    We call a reflection through a non-isotropic line an \emph{axial symmetry}.
    The coset of $\SF$ contains, in particular, the set of axial symmetries.
    
    We define axial symmetries relative to non-isotropic lines only, since if $\ell$ is isotropic, then $a$ being symmetric to $a'$ relative to $\ell$ means that $a-a'$ is normal to $\ell$, and also that for any $b\in \ell$, $d(a,b) = d(a',b)$.
    However, if $\ell$ is isotropic, this means that $a,a'$ lie on $\ell$.
    
    For $x,y\in \F^2$, we write $x\sim_\ell y$ to mean that $x$ is axially symmetric to $y$, relative to the (non-isotropic) line $\ell$. 
    
    The composition of two axial symmetries, relative to distinct lines $\ell$ and $\ell'$, as in the Euclidean case, is generally a rotation around the axes intersection point, by twice the angle between the lines. If the lines are parallel, it is a translation in the normal direction (note that $\ell,\ell'$ are non-isotropic lines).
    
    In the sequel, for convenience of working within the group structure of $\SF$, rather than its other coset, we map the set of all axial symmetries into the group $\SF$.
    We map an axial symmetry to $\SF$ by composing it with the fixed axial symmetry $\rho$ relative to a non-isotropic subspace $\ell_\tau$.
    
    The image of the set of axial symmetries under this mapping is the set of rotations around all points on $\ell_\tau$, which we denote by $R_\tau$.
    
    If $\ell_\tau$ is the $x$-axis, then explicitly
    \[
        R_\tau = \left\{\begin{pmatrix}
            u&-v& x_0(1-u)\\
            v&u&-x_0v\\
            0&0&1
        \end{pmatrix} \colon u^2 + v^2 = 1, u,v,x_0\in \F
        \right\}\,.
    \]
    A short calculation shows that, for this choice of $\ell_\tau$, the image of $R_\tau$ under the kinematic mapping is contained in the plane $X_2=0$.
    By Proposition~\ref{prop:kinematic-equivariant}, we see that $R_\tau$ is contained in a plane for any choice of $\ell_\tau$.
    This transformation motivates the role of incidence geometry.

\subsection{Incidence Geometry}\label{sec:incidences}
    The key tool that we will use to estimate $\pind{A}$ is an incidence bound between points and planes in $\mathbb{FP}^3$ by the second author \cite{rudnev2014number}; for a selection of applications of this bound, see the second author's survey \cite{rudnev2018point-plane}.
    
    \begin{thm}[Points-Planes in $\mathbb{FP}^3$] \label{t:rudnev}
        Let $\mathcal{P}$ be a set of points in $\F^3$ and let $\Pi$ be a set of planes in $\mathbb{FP}^3$, with $|\P|\leq |\Pi|$. 
        If $\mathbb{F}$ has positive characteristic $p$, suppose that  $|\P|\ll p^2$. 
        Let $k$ be the maximum number of collinear points in $\mathcal{P}$. Then
        \[
            \I(\mathcal{P},\Pi)\ll |\P|^{1/2}|\Pi| +k |\Pi|.
        \] 
    \end{thm}
    
    The proof of Theorem~\ref{t:main-general} proceeds by first relating the quantity $\pind(A)$ to the count of isosceles triangles. We count the number of isosceles triangles by studying the set of bisectors determined by $A$ considering the axial symmetries relative to this line set. Using the \BG{} embedding, we rephrase this as an incidence bound between points and planes. 
    
    \section{Bisector energy and isosceles triangles}\label{sec:bisect-energy-isosc}
    
    Let $\mathcal T =\mathcal T(A)$ be the number of \emph{non-degenerate isosceles triangles}\/ with vertices in $A$.
    A non-degenerate isosceles triangle means a triple $(a,b,b')$ with $d(a,b)=d(a,b')$ and the base  $b-b'$ non-isotropic.
    We use $\T(A)$ to denote the number of such triangles with $d(a,b),d(a,b')\not=0$.
    
    The number of isosceles triangles determined by $A$ is inversely proportional (from below) to the number of pinned distances determined by $A$.
    \begin{lem}
      \label{lem:pin-iso-lower-1}
    If $A$ is a subset of $\F^2$ with at most $M$ points on a line, then
    \[
    |A|(|A|-2M+1)^2\leq (\pind(A)+1)(\T(A)+|A|^2).
    \]
    \end{lem}
    Lund and Petridis prove this lemma as part of the proof of their main theorem \cite[p.10]{lund2018pinned}, but we give the proof here, since it is fundamental.
    \begin{proof}
    Write $C_r$ for the set of points in $\F^2$ of distance $r$ from the origin, and write $\Delta(A,a)$ to denote the set of non-zero distances determined by $a$.
    Then
    \[
    |A|(|A|-2M+1)\leq |A||A\setminus (a+C_0)| = \sum_{a\in A}\sum_{r\in \Delta(A,a)}|A\cap (a+C_r)|,
    \]
    so by Cauchy-Schwarz and the bound $|\Delta(A,a)|\leq\pind(A)+1$, we have
    \[
    |A|(|A|-2M+1)^2\leq (\pind(A)+1) \left( \sum_{a\in A}\sum_{r\not=0} |A\cap (a+C_r)|^2\right).
    \]
    We have
    \begin{equation}
      \label{eq:5}
      \sum_{a\in A}\sum_{r\not=0} |A\cap (a+C_r)|^2=\T(A)+|A|^2,
    \end{equation}
    since the sum on the left hand side is equal to $\T(A)$ plus the number of triples $(a,b,b')$ in $A^3$ such that $d(a,b)=d(a,b')\not=0$ and $b-b'$ is isotropic.
    If $b\not=b'$, then by Lemma 6 of Lund and Petridis \cite{lund2018pinned}, we have $d(a,b)=d(a,b')=0$, which is a contradiction, so there are $|A|^2$ such triples.
    \end{proof}

    Let $i_A(\ell)=|A\cap \ell|$ denote the number of points of $A$ incident to the line $\ell$, and let $b_A(\ell)$ denote the number of pairs of points $a,b\in A$ whose perpendicular bisector $\bisector(a,b)$ is $\ell$.
    The quantity $b_A(\ell)$ is equal to the number of points $a$ in $A$ such that there exists an $a'$ in $A$ that is symmetric to $a$ with respect to $\ell$.
    We introduce a modified quantity $b^*_A(\ell)$, which counts the number of such $a$ \emph{outside of $\ell$}:
    \[
        b^*_A(\ell):=|\{a\in A\setminus\ell\colon (\exists a'\in A)(a'\sim_\ell a)\}|.
    \]
    If $A$ contains an isotropic vector $a'\not=0$ with $d(a',a')=0$, then for all $a\in\ell$, the perpendicular bisector of $a$ and $a'$ is $\ell$.
    So if $N=|A\cap C_0|-1$ is the number of isotropic vectors in $A$, we have
    \[
        b_A(\ell)=i_A(\ell)N+b^*_A(\ell).
    \]
    
    The \emph{bisector energy}\/ of a set $A$ in $\F^2$ is the second moment of $b_A$:
    \[
        \B(A):=|\{(a,b,a',b')\in A^4\colon \bisector(a,b)=\bisector(a',b')\}|=\sum_\ell b_A(\ell)^2.
    \]
    We write $\B^*(A)$ for the second moment of $b_A^*(\ell)$; this modified bisector energy allows us to avoid pathologies arising from isotropic vectors.
    
    Our next lemma bounds $\T(A)$ in terms of $\B^*(A)$.\fxnote{we need to do something about the case where $A$ contains isotropic points, since $b_A(\ell)$ may be quite large}
    \begin{lem}  \label{lem:bisectors-control-triangles}
        If $A$ is a subset of $\F^2$, then
        \[
            \T(A) \leq |A|\B^*(A)^{1/2}.
        \]
    \end{lem}
    \begin{proof}\fxnote{I am using isotropic to mean $d(b,0)=0$, including $b=0$}
        Let $(a,b,b')$ be a non-degenerate triple, so that $d(a,b)=d(a,b')\not=0$ and $b-b'$ is not isotropic.
        Since $d(b,b')\not=0$, the perpendicular bisector $\ell=\B(b,b')$ is well defined, $a\in\ell$, and $b,b'\not\in\ell$. \fxnote{do we discuss perpendicular bisectors somewhere? we include the fact that they are not isotropic in 2.3...}
        Thus
        \begin{equation}
          \label{eq:T-sum}
          \T(A)=\sum_{\ell}i_A(\ell)b_A^*(\ell).
        \end{equation}
        To complete the proof, we apply the Cauchy-Schwarz inequality. 
    \end{proof}
      
    Let $\Q$ denote the number of non-zero distance quadruples:
    \[
      \Q(A):=|\{(a,b,a',b')\in A^4\colon d(a,b)=d(a',b')\not=0\}|,
    \]
    which satisfies the following equation
    \begin{equation}
      \label{eq:quadruples}
        \Q(A)= \sum_{r\not=0} \left( \sum_{a\in A}|A\cap (a+C_r)| \right)^2.
    \end{equation}
    Our first main technical result bounds the bisector energy of $A$ in terms of $\Q(A)$.
    \begin{prop}
      \label{prop:bisectors}
        Let $\F$ be a field of characteristic $p$.
        Suppose that $A\subseteq\F^2$ has cardinality $|A|\leq p^{4/3}$ and let $M$ denote the maximum number of collinear or co-circular points of $A$.
        Then the bisector energy of $A$ satisfies
        \[
            \B^*(A)\ll M|A|^2 + |A| \Q(A)^{1/2}.
        \]
    \end{prop}
    By Lemma~\ref{lem:bisectors-control-triangles}, this implies a bound on $\T(A)$.
    There are sets $A$ with $\B^*(A)\gg M|A|^2$ \cite[Section 3.4]{lund2016bisector}, however we can remove the dependence on $M$ from the bound on $\T(A)$; this is our second main technical result.
    \begin{prop}
      \label{prop:isosceles}
        Let $\F$ be a field of characteristic $p$.
        If $A\subseteq\F^2$ has cardinality $|A|\leq p^{4/3}$, then the number of non-degenerate isosceles triangles determined by $A$ satisfies
        \[
            \T(A) \ll |A|^{7/3}.
        \]
    \end{prop}
    Theorem~\ref{t:main-general} follows immediately from this proposition and Lemma~\ref{lem:pin-iso-lower-1}.
    
    The following formula bounds the number of distance quadruples in terms of the number of isosceles triangles:
    \begin{equation}
      \label{eq:Q-less-than-A-T}
        \Q(A)= \sum_{r\not=0} \left( \sum_{a\in A}|(a-A)\cap C_r| \right)^2 
        \leq |A|\sum_{r\not=0} \sum_{a\in A}|(a-A)\cap C_r|^2= |A|\T(A).
    \end{equation}
    Combined with Proposition~\ref{prop:isosceles}, equation~\eqref{eq:Q-less-than-A-T} implies a new bound on the number of distance quadruples.
    \begin{cor}
      \label{cor:quad-10-3}
        If $\F$ is a field of characteristic $p$ and $A\subseteq\F^2$ has cardinality $|A|\leq p^{4/3}$, then
        \[
            \Q(A)\ll |A|^{10/3}.
        \]
    \end{cor}
    Equation~\eqref{eq:Q-less-than-A-T} plays an important role in the proof of Proposition~\ref{prop:isosceles}. We thank Giorgis Petridis and Thang Pham for pointing out that \eqref{eq:Q-less-than-A-T} could be used to improve our original proof, resulting in a lower bound of $\pind(A)\gg |A|^{2/3}$ for $|A|\leq p^{4/3}$.
    From this observation, we are able to prove the slightly stronger energy estimate of Proposition~\ref{prop:isosceles}.

\section{Proof of Proposition~\ref{prop:bisectors}}
    Let $S_r\subseteq A\times A$ denote the set of segments of length $r$:
    \[
        S_r:=\{(a,a')\in A\times A\colon d(a,a')=r\}.
    \]
    We have 
    \[
        |S_r|=\sum_{a\in A}|A\cap (a+C_r)|,
    \]
    so
    \[
        \sum_{r\not=0}|S_r|=|A||A\setminus C_0|\andd \sum_{r\not=0}|S_r|^2 = \Q(A).
    \]
    
    Let $\mathrm{Ax}_{(c,d)}$ be the set of elements $(x,y)\in \F^2$ that are axially symmetric to $(c,d)\in \F^2$ (with respect to some non-isotropic line).
    For a set $X\subseteq A\times A$ containing no isotropic segments (that is, $\norm{a-b}\not=0$ for all $(a,b)\in X$, let $\mathcal{A}(X) := \{\mathrm{Ax}_x\colon x\in X\}$ be the set of elements attainable from $X$ via axial symmetries. 
    Then, letting $\calL_\perp(A)$ denote the set of non-isotropic perpendicular bisectors of $A$, we have
    \begin{align*}
    \B^*(A) &=\sum_\ell b_A^*(\ell^2)\\
       &=|\{(a,b,c,d,\ell)\in A^4\times \calL_\perp(A) \colon (a,b)\sim_\ell (c,d)\}|\\
       & = \sum_{r\neq 0}|\{((a,b),(c,d)\in S_r^2 \colon (a,b)\in\mathrm{Ax}_{(c,d)}\}| + |\{((a,b),(c,d)\in S_0^2 \colon (a,b)\sim_\ell(c,d)\}|\\
       &=\sum_{r\neq 0}\I(S_r, \mathcal{A}(S_r)) + \calE,
     \end{align*}
    where the last line is a definition.
    
    The error term $\calE$ can be bounded by $2M|A|^2$, since for each $a$, there are at most $2M$ choices of $b$ such that $a-b$ is isotropic; if $a$ and $b$ are chosen, and $c$ is axially symmetric to $a$, then there is only one choice for $d$, which gives the claimed bound.
    
    Let us prove the following claim.
    \begin{claim}
        \label{claim}
        Let $r\neq 0$, and suppose, if $\F$ has positive characteristic $p$, that $|A|\leq p^{4/3}$.
        Suppose that at most $M$ points of $A$ are collinear or co-circular in $\F^2$.
        Then
        \[
            \I(S_r,\mathcal{A}(S_r))\ll M|S_r| + |S_r|^{3/2}.
        \]
    \end{claim}

    \begin{proof}[Proof of Claim~\ref{claim}]
        Passing to an extension of $\F$ can only increase the quantity we seek to bound, so we assume, without loss of generality, that $\F$ is algebraically closed.
        
        We embed the set $S_r$ in $\SF$ by fixing a segment $s_r$ in $S_r$ and identifying an element $(a,a')\in S_r$ with the inverse of the rigid motion that takes $s_r$ to $(a,a')$.
        This rigid motion always exists, for one can translate $a)$ to the origin, and then find the corresponding rotation, for $r\neq 0$. 
        Let $G_r$ denote the set of transformations in $\SF$ corresponding to segments in $A^2$.

        Now we will associate a \emph{projective plane}\/ in \PF{} to each segment in $S_r$.
        Choose $\tau$ so that for all $g,h\in G_r$, the transformation $g^{-1}h$ has no fixed points on $\ell_\tau$; this is possible since $\F$ is algebraically closed, so there are infinitely many choices of $\ell_\tau$, while there are a finite number of products $g^{-1}h$.
        Recall that $R_\tau$ is the set of axial symmetries composed with a reflection about the non-isotropic subspace $\ell_\tau$ and that $\kappa(R_\tau)$ is contained in a projective plane $\kappa(R_\tau)$.
        Let $g$ be the element of $\SF$ corresponding to $(a,a')$.
        By Proposition~\ref{prop:kinematic-equivariant}, the transformation $\phi_g$ is projective, hence the set $\phi_g(\kappa(R_\tau))$ is a projective plane in $\PF$.
        Let $\Pi=\{\phi_g(\kappa(R_\tau))\colon g\in G_r\}$.
        We have $|\Pi|=|G_r|=|S_r|$, since $\phi_g(\kappa(R_\tau))=\phi_h(\kappa(R_\tau))$ if and only if $g^{-1}h\in R_\tau$, but every element of $R_\tau$ fixes a point on $\ell_\tau$, while no product $g^{-1}h$ with $g$ and $h$ in $G_r$ fixes a point on $\ell_\tau$.
        
        Let $G_r'$ denote the set of $g\in\SF$ such that $g^{-1}s_r\in\tau(A)\times\tau(A)$, and set $P=\kappa(G_r')$.
        We will show that
        \[
            \I(S_r,\mathcal{A}(S_r)) = \I(P,\Pi).
        \]
        First note that $|P|=|G_r'|=|S_r|$, since the kinematic mapping is injective.
        Now, suppose that $\pi=\phi_g(\kappa(R_\tau))$ for some $g \in G_r$ and $p=\kappa(h)$ for some $h\in G_r'$.
        If $p\in\pi$, then $\kappa(h)\in \phi_g(\pi)=\phi_g(\kappa(R_\tau))$.
        Thus
        \[
            \kappa(g^{-1}h)=\phi_g^{-1}(\kappa(h))\in\kappa(R_x),
        \]
        so $h\in gR_x$.
        Now, let $(a,a')$ correspond to $g$ and $(b,b')$ correspond to $h$, so that
        \[
            g(a,a')=h(\tau(b),\tau(b'))=s_r.
        \]
        Since $h\in gR_x$, we have
        \[
            (\tau(b),\tau(b'))\in R_x^{-1}(a,a')=R_x(a,a'),
        \]
        thus $(b,b')$ is attainable from $(a,a')$ by an axial symmetry.
        
        We apply Theorem~\ref{t:rudnev} to $P$ and $\Pi$, claiming that the number of collinear  points or planes is bounded by $M$. This is a direct consequence of  Lemma 5 of Lund and Petridis \cite{lund2018pinned}, which states that given two segments $s,s'$ of given length $r$, the endpoints of every segment $s''$, axially symmetric to both $s,s'$ lies on a pair of concentric circles or parallel lines, uniquely defined by $s,s'$, whose endpoints also lie on this pair of circles/lines.
        
        Thus
        \[
            \I(S_r,\mathcal A(S_r) ) = \I(P,\Pi) \ll M|S_r|+|S_r|^{3/2}\,.
        \]
        If $\F$ has positive characteristic $p$, then also we need the estimate $|\Pi|\ll p^2$; since $|\Pi|=|S_r|\ll |A|^{3/2}$ by Erd\H{o}s' bound on the number of times a distance can repeat \cite{erdos1946distances}, we have the required constraint for $|A|\ll p^{4/3}$.
        
        This completes the proof of the Claim.
    \end{proof}

\section{Proof of Proposition~\ref{prop:isosceles}}
\label{sec:proofmain}
    The proof of Proposition~\ref{prop:isosceles} essentially combines Proposition~\ref{prop:bisectors} with equation~\eqref{eq:Q-less-than-A-T}.
    We thank Giorgis Petridis and Thang Pham, who pointed out this idea; previously we used a trivial bound for $\Q(A)$, and improved this bound in some cases by other methods.
    Their observation immediately implied that $\pind(A)\gg |A|^{2/3}$ under the assumptions of Theorem~\ref{t:main-general}; subsequently we improved this to an unconditional bound on $\T(A)$ (and hence on $\Q(A)$, the number of distance quadruples).
    
    In order to get a bound for $\T(A)$ that is independent of $M$, the number of collinear or co-circular points of $A$, we need an addition argument to remove rich lines and circles.
    \begin{lem}[Pruning heavy circles and lines]
      \label{lem:3}
        Suppose that $A$ is the disjoint union of $B$ and $C$.
        If all of the points of $C$ are contained in a circle or a line (denoted $\gamma$), then
        \[
            \T(A)\leq \T(B) + 8|A|^2.
        \]
    \end{lem}
    \begin{proof}
        We have
        \[
            i_A(\ell) = i_B(\ell)+i_C(\ell)
        \]
        and, if $r_\ell$ is the reflection through $\ell$,
        \begin{align*}
            b_A^*(\ell)&=|r_\ell(A\setminus\ell)\cap A|\\
            &\leq b_B^*(\ell)+b_C^*(\ell)+|r_\ell(B\setminus \ell)\cap C|+|r_\ell(C\setminus \ell)\cap B|\\
            &= b_B^*(\ell)+b_C^*(\ell)+2|r_\ell(C\setminus \ell)\cap B|.
        \end{align*}
        Thus
        \begin{align*}
            \T(A)&= \T(B)+\sum_\ell i_C(\ell)b_A^*(\ell) +\sum_\ell i_B(\ell)b_C^*(\ell) +2\sum_\ell i_B(\ell)|r_\ell(C\setminus\ell)\cap B|\\
            &:=\T(B)+I+II+III.
        \end{align*}
        
        If $\gamma$ is a circle, then $I\leq 2|A|^2$, since $i_C(\ell)\leq 2$
        \[
            \sum_\ell b_A^*(\ell)\leq \#(\mbox{perpendicular bisectors of points in $A$})\leq |A|^2.
        \]
        If $\gamma$ is a line, then $i_C(\ell)\leq 1$ except for $\ell=\gamma$, so
        \[
            I \leq |A|^2 +|C||A|\leq 2|A|^2.
        \]
        
        To bound $II$, notice that $b_C^*(\ell)$ is zero, unless $\ell$ goes through the center of $\gamma$ (or unless $\ell$ is perpendicular to $\gamma$); except for the center of $\gamma$, each point of $B$ is incident to at most one line $\ell$ for which $b_C^*(\ell)\not=0$.
        Using the trivial bound $b^*_C(\ell)\leq |C|$, we have
        \[
            II\leq |B||C|+|C|^2,
        \]
        where the second term covered the case where $B$ contains the center of $\gamma$.
        
        To bound $III$, we use that two distinct points can be incident to at most one line and at most two circles of the same radius.
        
        Suppose that $\gamma$ is a circle.
        Let
        \[
            X:=\sum_\ell i_B(\ell)|r_\ell(C\setminus\ell)\cap B|.
        \]
        By Cauchy-Schwarz,
        \[
            X\leq |B| \left( \sum_\ell |r_\ell(C\setminus\ell)\cap B|^2 \right)^{1/2}.
        \]
        The sum in parenthesis is equal to
        \[
            Y:=|\{(\ell,b,b')\colon b,b'\in B, i_B(\ell)>0, b,b'\in r_\ell(C\setminus\ell)\}|.
        \]
        Since $B$ and $C$ are disjoint, we may assume that $r_\ell$ does not fix $\gamma$, so for $b\not=b'$, there are at most two lines $\ell$ so that $b,b'\in r_\ell(C)$.
        Thus
        \[
            Y \leq 2|B|^2 + \sum_{\ell\colon i_B(\ell)>0}|r_\ell(C\setminus\ell)\cap B|\leq 2|B|^2+X.
        \]
        Returning to our initial equation for $X$, we have
        \[
            X^2 \leq |B|^2X+4|B|^4,
        \]
        so $X\leq 4|B|^2$.
        
        Now suppose that $\gamma$ is a line.
        If $r_\ell(\gamma)=\gamma$, then $\gamma$ and $\ell$ are perpendicular (or $\ell=\gamma$).
        Let $X$ be as above, and define
        \[
            X^*:=\sum_{\ell\not\perp\gamma}i_B(\ell)|r_\ell(C\setminus\ell)\cap B|.
        \]
        Since
        \[
            \sum_{\ell\perp\gamma}i_B(\ell)|r_\ell(C\setminus\ell)\cap B| \leq |B||C|,
        \]
        we have $X\leq X^*+|B||C|$.
        
        Let $Y^*$ denote the corresponding quantity arising from Cauchy-Schwarz
        \[
            Y^*:=|\{(\ell,b,b')\colon b,b'\in B, i_B(\ell)>0, \ell\not\perp\gamma, b,b'\in r_\ell(C\setminus\ell)\}|.
        \]
        We may assume that $\ell\not=\gamma$, since then $C\setminus\ell$ is empty, so $r_\ell$ does not fix $\gamma$.
        If $b$ and $b'$ are distinct, then there is at most one $\ell$ so that $b,b'\in r_\ell(\gamma)$.
        Thus
        \[
            X^* \leq |B|\sqrt{Y^*} \leq |B| \left( |B|^2+X^* \right)^{1/2},
        \]
        so $X^*\leq 2|B|^2$ and hence $X\leq 3|B|^2$.
        
        Combining these cases, we see that
        \[
            \T(A)\leq \T(B)+8|A|^2,
        \]
        as claimed.
    \end{proof}
    
    \begin{lem}[Pruning heavy lines and circles]
      \label{lem:1}
        There is a subset $A'\subseteq A$ such that at most $|A|^{2/3}$ points of $A'$ are collinear or co-circular, and
        \[
            \mathcal T(A)\leq \mathcal T(A') + 8|A|^{7/3}.
        \]
    \end{lem}
    \begin{proof}
        Use Lemma~\ref{lem:3} to greedily remove lines and circles, gaining a factor of $8|A|^2$ each time.
        If we only remove lines and circles with more than $|A|^{2/3}$ points on them, then this procedure terminates after $|A|^{1/3}$ steps.
    \end{proof}
    
    Let $M:=|A|^{2/3}$ denote the number of points of $A'$ that are collinear or coplanar.
    By Lemma~\ref{lem:bisectors-control-triangles},
    \[
        \sum_{\ell\in \calL_\perp(A')} i_{A'}(\ell)b_{A'}^*(\ell) \leq 2|A'|^2+|A'|\B^*(A')^{1/2}.
    \]
    By Proposition~\ref{prop:bisectors}, if $|A'|\leq p^{4/3}$, then
    \[
        \B^*(A')\ll M|A'|^2 +|A'|\Q(A')^{1/2} \leq |A|^{8/3}+|A|\Q(A)^{1/2},
    \]
    so
    \[
        \T(A)\ll |A|^{7/3} + \T(A')
    % \ll |A|^{7/3}+|A| \left( |A|^{8/3}+|A|\Q(A)^{1/2} \right)^{1/2}
        \ll |A|^{7/3}+|A|^{3/2}\Q(A)^{1/4}.
    \]
    By equation~\eqref{eq:Q-less-than-A-T}, we have
    \[  
        \T(A) \ll |A|^{7/3} + |A|^{7/4}\T(A)^{1/4},
    \]
    so $\T(A)\ll |A|^{7/3}$, as desired.

\appendix

\section{Clifford algebra computations}
\label{sec:clifford-calc}
    This section is a short digest of Clifford algebras over finite fields.
    We follow Klawitter and Hagemann \cite{klawitter2013kinematic}, who give a similar exposition for Clifford algebras over $\R$.
    We are indebted to Jon Selig, who told us about the connection between Clifford algebras and the \BG{} kinematic mapping.
    
    For a vector space $V$ with a quadratic form $Q$, the Clifford algebra $\cl(V,Q)$ is the largest algebra containing $V$ and satisfying the relation that $x^2=Q(x)$ for all $x\in V$, where $x^2$ is the square of $x$ in the algebra.
    If $V$ is an $n$-dimensional vector space over a finite field $\F$ of odd characteristic, then there is a basis $e_1,\ldots,e_n$ of $V$ such that $Q(e_i)=\lambda_i$, where $\lambda_i$ is one of: 0,1,a non-square.
    This basis is orthogonal with respect to the bilinear form associated to $Q$.
    The Clifford algebra is a $2^n$-dimensional $\F$ vector space with basis $e_{i_1\ldots i_k}$ where $i_1<\cdots<i_k$ and $0\leq k\leq n$ (with the understanding that the empty index is $e_0$) defined by
    \[
        e_{i_1\ldots i_k}=e_{i_1}\cdots e_{i_k}.
     \]
    The rules for multiplication in $\cl(V,Q)$ are given by $e_ie_j=e_je_i$ for $i\not=j$ and $e_i^2=\lambda_i$.
    
    The Clifford algebra $\cl(V,Q)$ splits as a direct sum of exterior products
    \[  
        \cl(V,Q)=\bigoplus_{i=0}^n \bigwedge^i V
    \]
    and is $\Z/2$-graded:
    \[
        \cl(V,Q)=\cl(V,Q)^+\oplus \cl(V,Q)^-,
    \]
    where
    \[
        \cl(V,Q)^+:=\bigoplus_{\substack{i=0\\ i\equiv 0\pmod 2}}^n \bigwedge^i V
            \andd
        \cl(V,Q)^-:=\bigoplus_{\substack{i=0\\ i\equiv 1\pmod 2}}^n \bigwedge^i V.
    \]
    The dimension of the even subalgebra $\cl(V,Q)^+$ is $2^{n-1}$. \fxnote{Klawitter and Hagemann give an isomorphism between the even subalgebra and another universal Clifford algebra, but this must be different in our case...}
    We identify $\bigwedge^0 V$ with $\F$ and $\bigwedge^1 V$ with $V$.
    
    We define two automorphisms of $\cl(V,Q)$.
    The first, called \emph{conjugation}, is denoted by an asterisk.
    For the basis elements of $V$ we define conjugation by $e_i^*=-e_i$.
    We extend conjugation to other basis elements by changing the order of multiplication
    \[
        (e_{i_1}e_{i_2}\cdots e_{i_k})^*:=(-1)^ke_{i_k}\cdots e_{i_2}e_{i_1}\quad 0\leq i_1<i_2<\cdots<i_k\leq n.
    \]
    Finally, we extend conjugation to $\cl(V,Q)$ by linearity.
    One can check that $(\fa\fb)^*=\fb^*\fa^*$ for an elements $\fa,\fb\in\cl(V,Q)$.
    (Notice that if $\fa\in\bigwedge^k V$, then $\fa^*=(-1)^{k(k+1)/2}\fa$.)
    We define the norm of an element $\fa$ by $N(\fa)=\fa\fa^*$.
    
    The second automorphism of $\cl(V,Q)$, called the \emph{main involution}, is denoted by $\alpha$ and defined by $\alpha(e_i)=-e_i$ and extended to $\cl(V,Q)$ by linearity and the rules for multiplication.
    Clearly $\alpha$ fixes the even subalgebra $\cl(V,Q)^+$ and acts by multiplication by $-1$ on the odd subalgebra $\cl(V,Q)^-$.

    Let $\cl^\times(V,Q)$ denote the set of invertible elements of $\cl(V,Q)$, which we call \emph{units}.
    If $N(\fa)=1$, then $\fa\in\cl^\times(V,Q)$.
    If $\fa\in V$ and $N(\fa)\not=0$, then $\fa\in\cl^\times(V,Q)$, and $\fa^{-1}=\fa^*/N(\fa)$.
    The \emph{Clifford group}\/ associated to $\cl(V,Q)$ is defined by
    \[
        \Gamma(\cl(V,Q)):=\left\{\fg\in\cl^\times(V,Q)\colon \alpha(\fg)V\fg^{-1}\subseteq V \right\}.
    \]
    We say that the map $\fv\mapsto \alpha(\fg)\fv\fg^{-1}$ is the \emph{sandwich operator}\/ associated to an element $\fg\in\Gamma(\cl(V,Q))$.
    
    Given a quadratic form $Q_0$ on $\F^2$ with $Q_0(e_1)=1$ and $Q_0(e_2)=-\lambda$, let $\SO$ denote the group of rotations preserving $Q_0$:
    \[
        \SO:= \left\{
        \begin{pmatrix}
            u & v \\
            \lambda v & u\\
        \end{pmatrix}
        \colon u^2-\lambda v^2=1
        \right\},
    \]
    and let $\SF$ denote the group of rigid motions of $\F^2$ generated by $\SO$ and the group of translations.
    \begin{prop}
      \label{prop:1}
        Let $V=\F^3$ and define $Q$ on $V$ by $Q(x,y,z)=Q_0(x,y)$, let $G=(\cl(V,Q)^+)^\times$ be the group of units of the even subalgebra, and let $Z$ be its centre.
        Then $G/Z$ is isomorphic to $\SF$.
    \end{prop}
    \begin{proof}
        By our definition of $Q$, $e_1^2=1,e_2^2=-\lambda,e_3^2=0$, $\cl(V,Q)$ is spanned by
        \[
            e_0, e_1,e_2,e_3,e_{12},e_{13},e_{23},e_{123},
        \]
        and $\cl(V,Q)^+$ is spanned by $e_0,e_{12},e_{13},e_{23}$.
    % Further,
    % \[
    % e_{12}^2=\lambda, e_{13}^2=0, e_{23}^2=0, e_{12}e_{13}=-e_{23}, e_{12}e_{23}=-\lambda e_{13},e_{13}e_{23}=0,
    % \]
    % and $e_{13}e_{12}=-e_{12}e_{13}$ and $e_{23}e_{12}=-e_{12}e_{23}$. 
        If $\fg= g_0e_0 + g_{12}e_{12}+g_{13}e_{13}+g_{23}e_{23}$, then
        \[
            N(\fg)=\fg\fg^*=g_0^2-\lambda g_{12}^2\,.
        \]
        Thus, if $g_0^2-\lambda g_{12}^2\not=0$, the inverse of $\fg$ is 
        \[
            \fg^{-1}=\frac 1{g_0^2-\lambda g_{12}^2}\fg^*.
        \]
        This determines the group of units explicitly.
    
        One can show by a computation that $G$ acts on $V$ by the sandwich product $(\fg,\fv)\mapsto \fg\fv\fg^{-1}$ (that is $G=\Gamma(\cl(V,Q)^+)$).
        In fact, the action of general element $\fg=g_0e_0 +g_{12}e_{12}+g_{13}e_{13}+g_{23}e_{23}$ in $G$ is given by
        \begin{align*}
            \fg e_1 \fg^{-1}&=\frac{g_0^2+\lambda g_{12}^2}{g_0^2-\lambda g_{12}^2}e_1
            +\frac{-2g_0g_{12}}{g_0^2-\lambda g_{12}^2}e_2
            +\frac{-2(g_0g_{13}+\lambda g_{12}g_{23})}{g_0^2-\lambda g_{12}^2}e_3,\\
            \fg e_2 \fg^{-1}&=\frac{-2\lambda g_0g_{12}}{g_0^2-\lambda g_{12}^2}e_1
            +\frac{g_0^2+\lambda g_{12}^2}{g_0^2-\lambda g_{12}^2}e_2
            +\frac{2\lambda(g_0g_{23}+ g_{12}g_{13})}{g_0^2-\lambda g_{12}^2}e_3,\\
            \fg e_3 \fg^{-1}&=e_3.
        \end{align*}
        Let $\rho\colon G\to\GL(V)$ denote this representation.
        The dual representation $\rho^*(\fg):=\rho(\fg^{-1})^T$, where $T$ denotes the transpose, acts on the dual space $V^*$, and in the standard basis $\{f_1,f_2,f_3\}$ on $V^*$ defined by $f_i(e_j)=\delta_{ij}$, we have
        \begin{equation}
          \label{eq:contragredient}
            \rho^*(\fg^{-1})=
            \frac 1{g_0^2-\lambda g_{12}^2}
            \begin{pmatrix}
                g_0^2+\lambda g_{12}^2 &-2g_0g_{12} & -2(g_0g_{13}+\lambda g_{12}g_{23})\\
                -2\lambda g_0g_{12} & g_0^2+\lambda g_{12}^2 & 2\lambda(g_0g_{23}+ g_{12}g_{13})\\
                0 & 0&  g_0^2-\lambda g_{12}^2\\
            \end{pmatrix}.
        \end{equation}
        The kernels of $\rho$ and $\rho^*$ are both equal to the subgroup $Z:=\{g_0e_0\colon g_0\not=0\}$.
        We wish to show that $G/Z$ is isomorphic to $\SF$.
        Let $R$ be the subgroup defined by $g_{13}=g_{23}=0$; the rational parameterisation of the circle shows that $\rho^*(R)$ is in bijection with the subgroup $\SO\subseteq\SF$.
        On the other hand, it is clear that the subgroup $T$ defined by $g_0=1,g_{12}=0$ is bijective with the translation subgroup of $\SF$.
        Since these subgroups generate $\SF$, we see that $\rho^*$ is surjective.
    \end{proof}
    
    We have shown more: $\SF$ is naturally identified with an (open) subset of $\PF$, and the nature of this identification yields some desirable features.
    In particular, the set of transformations in $\SF$ that map a point $x\in \F^2$ to a point $y\in\F^2$ is a line. 
    
    Let $\kappa\colon\SF\to G/Z$ denote the inverse of $\rho^*\colon G/Z\to\SF$.
    This is the \emph{kinematic mapping}\/ of Blaschke and Gr\"unwald, who both sought to embed the group of rigid motions in projective space.
    Let $\PF$ denote projective three space; we write $[X_0\colon X_1\colon X_2\colon X_3]$ for a typical point of $\PF$.
    \begin{cor}
      \label{cor:1}
        There is a bijection $\kappa\colon\SF\to\PF\setminus\{X_0^2-\lambda X_1^2=0\}$ such that the image of the rotation subgroup and translation subgroups are projective lines.
    
        Further, for all $g\in\SF$ there are projective maps $\phi_g\colon \PF\to \PF$ and $\phi^g\colon \PF\to \PF$ such that for all $x\in\SF$
        \[
            \kappa(gx) =\phi_g(\kappa(x))\andd \kappa(xg)=\phi^g(\kappa(x)).
        \]
    \end{cor}
    \begin{proof}
        The even subalgebra $\cl(V,Q)^+$ is isomorphic to $\F^4$ as a vector space, so the projective space $\PP(\cl(V,Q)^+)$ is $\PF$.
        On the other hand, $\PP(\cl(V,Q)^+)$ is just $\cl(V,Q)^+$ modulo the action of the multiplicative subgroup $Z$, so we have $G/Z\subseteq\PP(\cl(V,Q)^+)$.
        In fact, $G/Z$ consists of all points $[g_0\colon g_{12}\colon g_{13}\colon g_{23}]$ such that $g_0^2-\lambda g_{12}^2\not=0$.
        
        Since $\cl(V,Q)^+$ is an $\F$-algebra, left and right multiplication are $\F$-linear transformations.
        That is, if $\tilde\phi_{\fg}(\fv):=\fg\fv$ and $\tilde\phi^{\fg}(\fv):=\fv\fg$, then $\tilde\phi_{\fg}$ and $\tilde\phi^{\fg}$ are linear transformations.
        It follows that left and right translation in $G/Z$ are \emph{projective transformations}\/ of $\PF$.
    \end{proof}

\bibliographystyle{siam}
\bibliography{library}

\end{document}